\documentclass[12pt]{amsart}

\usepackage{amsmath, amssymb}
\usepackage{array}
\usepackage[frame,cmtip,arrow,matrix,line,graph,curve]{xy}
\usepackage{graphpap, color, paralist, pstricks}
\usepackage[mathscr]{eucal}
\usepackage[pdftex]{graphicx}
\usepackage[pdftex,colorlinks,backref=page,citecolor=blue]{hyperref}
\usepackage{tikz}
\usepackage{tikz-cd}

\usepackage{enumitem}
\usepackage{listings}

\newtheorem{theorem}{Theorem}[section]

\newtheorem{proposition}[theorem]{Proposition}
\newtheorem{corollary}[theorem]{Corollary}
\newtheorem{lemma}[theorem]{Lemma}

\theoremstyle{definition}
\newtheorem{definition}[theorem]{Definition}

\newtheorem{question}[theorem]{Question}
\newtheorem{remark}[theorem]{Remark}

\renewcommand{\AA}{\mathbb{A} }

\newcommand{\FF}{\mathbb{F} }
\newcommand{\PP}{\mathbb{P} }

\newcommand{\bp}{\mathbf{p} }

\newcommand{\bv}{\mathbf{v} }

\title{On algebraic space filling curves}

\author[A. Campbell]{Alana Campbell}
\address{Alana Campbell, Department of Mathematics, Fordham University, New York, NY 10023}
\email{acampbell50@fordham.edu}

\author[F. Dedvukaj]{Flora Dedvukaj}
\address{Flora Dedvukaj, Department of Mathematics, Fordham University, New York, NY 10023}
\email{fdedvukaj@fordham.edu}

\author[D. McCormick III]{Donald McCormick III}
\address{Donald McCormick III, Department of Mathematics, Fordham University, New York, NY 10023}
\email{dmccormick12@fordham.edu}

\author[H.-B. Moon]{Han-Bom Moon}
\address{Han-Bom Moon, Department of Mathematics, Fordham University, New York, NY 10023}
\email{hmoon8@fordham.edu}

\author[J. Morales]{Joshua Morales}
\address{Joshua Morales, Department of Mathematics, Fordham University, New York, NY 10023}
\email{jmorales77@fordham.edu}

\date{\today}

\begin{document}

\maketitle

\begin{abstract}
Poonen and Gabber independently showed that any smooth geometrically irreducible projective scheme over a finite field has a smooth space filling curve, that is, a smooth curve defined over the field and passes through all points over the field. However, except the case of projective plane, no concrete example was found in literature. In this note, we construct explicit examples of algebraic space filling curves in three dimensional projective space, in particular the ones with minimum degree. 
\end{abstract}

\section{Introduction}

Since Peano's construction \cite{Pea90} in the late 19-th century, there have been many examples of curves passing through every point of higher dimensional manifolds. In analysis, a \emph{space filling curve} of a manifold $M$ (possibly with boundaries) is a continuous surjective map $f : [0, 1] \to M$. By Hahn-Mazurkiewicz theorem, any connected compact manifold admits a space filling curve. 

As one may expect, a space filling curve cannot have simple geometric structure. If $\dim M \ge 2$, it is not injective (otherwise it will give a homeomorphism between $M$ and $[0, 1]$), it is not differentiable and has fractal nature. 

In algebraic geometry, we may explore vast new geometric spaces beyond topological manifolds over real or complex numbers. Many constructions and geometric intuitions for manifolds can be extended to schemes over arbitrary fields, including finite fields. On this generality, we can see new fascinating geometric phenomena. One of them is the existence of a \emph{smooth} space filling curve that is \emph{embedded} in a given smooth scheme over a finite field. 

Let $\FF_q$ be a finite field of order $q = p^r$, where $p$ is a prime number. Let $X$ be a projective smooth scheme defined over $\FF_q$. We denote the set of $\FF_q$-rational points by $X(\FF_q)$. Then $X(\FF_q)$ is a finite set. We say a curve $C \subset X$ is a \emph{space filling curve} over $\FF_q$ if $C(\FF_q) = X(\FF_q)$. 

In \cite{Kat99}, Katz constructed a smooth space filling curve for an affine space $\AA^n$ over $\FF_q$. His construction crucially depends on the existence of a high degree \'etale map $\AA^1 \to \AA^1$, which cannot be extended to the projective case. So he asked if one can construct a smooth space filling curve for any smooth geometrically integral projective schemes over $\FF_q$. Gabber and Poonen gave an affirmative answer independently \cite{Gab01, Poo04}. See the precise statement in Theorem \ref{thm:Poonen}. However, their proofs do not provide concrete examples, but show the existence. Furthermore, finding examples with low complexity (such as degree) is also an interesting question. 

In this paper, we focus on the question constructing explicit examples of space filling curves of a projective space. The simplest non trivial case is $\PP^2$. In \cite{HK13}, Homma and Kim constructed a smooth space filling curve $C \subset \PP^2$ for every $\FF_q$. Indeed, they verified that Tallini's example of an irreducible plane filling curve \cite{Tal61} is indeed smooth. Note that in this case, the space filling curve is a hypersurface, hence it is sufficient to find one equation.

The main result of this paper is providing examples of smooth space filling curves of $\PP^3$ over $\FF_q$ with small order $q$. 

\begin{theorem}\label{thm:mainthm}
Over $\FF_q$ with $q \le 7$, there is a complete intersection smooth space filling curve in $\PP^3$ of degree $(q+1)(q+2)$.
\end{theorem}

In Corollary \ref{cor:minimumdegree}, we show that $(q+1)(q+2)$ is the smallest possible degree of a complete intersection space filling curve. Thus, at least among complete intersections, our examples are optimal. 

Unfortunately, we were unable to prove the existence of such a minimal degree space filling curve for arbitrary $\FF_q$. However, we expect that such a curve always exist. We leave our attempt to prove it, and encountered challenge, in Section \ref{sec:existence}.

\begin{question}\label{que:minimaldegree}
Can we always find a smooth space filling curve in $\PP^3$ over $\FF_q$ of degree $(q+1)(q+2)$?
\end{question}

This paper is organized as follows. In Section \ref{sec:prerequisite}, we give a precise definition of a space filling curve and known example for $\PP^2$. We describe our basis strategy in Section \ref{sec:spacefillingsurface}. As a first step, it is important to study space filling surfaces. The section is a study of them. In Section \ref{sec:numericalexamples}, using a random construction with a computer algebra system, we give examples of space filling curves with low degree, and complete the proof of Theorem \ref{thm:mainthm}. Finally, in the last section, we describe our attempt to answer Question \ref{que:minimaldegree}.

\section{Schemes over finite fields and Space filling curves}\label{sec:prerequisite}

In this section, we review basic facts on algebraic geometry over finite fields and space filling curves. 

\subsection{Schemes over finite fields}\label{ssec:schemes}

Let $\FF_q$ be the finite field of order $q = p^r$, for some prime number $p$. A \emph{projective $\FF_q$-scheme} is a common zero set of finitely many homogeneous polynomials $f_1, f_2, \cdots, f_k \in \FF_q[x_0, \cdots, x_n]$. We denote it by $V(f_1, \cdots, f_k)$, and we have an embedding $V(f_1, \cdots, f_k) \subset \PP^n$. Using not necessarily homogeneous polynomials, we can define an \emph{affine $\FF_q$-scheme} in the same way, and we retain the same notation if there is no chance of confusion.

If we denote the ideal generated by $f_1, \cdots, f_k$ by $I$, then the associated scheme depends only on the ideal $I$, hence we use the notation $X = V(I)$. Conversely, for a scheme $X \subset \PP^n$, we denote its associate ideal as $I(X)$.

The distinction between the scheme and the set of points is important on algebraic geometry over finite fields. Let $X = V(f_1, \cdots, f_k)$ be a projective $\FF_q$-scheme. A point $\bp \in X$ is a point in the homogeneous coordinate $(p_0 : p_1 : \cdots : p_n)$ such that 
\begin{enumerate}
    \item $f_i(p_0, p_1, \cdots, p_n) = 0$ for all $1 \le i \le k$;
    \item all entries are in an extension field $K$ of $\FF_q$. 
\end{enumerate}
In this case, we say that $\bp$ is a \emph{$K$-rational point} (or simply $K$-point) of $X$. The set of $K$-rational points in $X$ are denoted by $X(K)$.

Because we are working over a finite field $\FF_q$, the set of all $\FF_q$-points in $\PP^n$ is a finite set, therefore, for any projective variety $X$, $|X(\FF_q)| < \infty$. Note that, however, if $\dim X > 0$, the set of all points over $X$ is infinite, as over the algebraic closure $\overline{\FF}_q$, $|X(\overline{\FF}_q)| = \infty$. 

Finally, for two projective schemes $X = V(I), Y = V(J)$ in $\PP^n$ over $\FF_q$, we say $X$ is a subscheme of $Y$ (or $X \subset Y$) if $J \subset I$. This implies $X(K) \subset Y(K)$ for all extension field $K$ of $\FF_q$. Note that, however, $X(K) \subset Y(K)$ for all $K$ does not imply $X \subset Y$.

\subsection{Existence of space filling curves}\label{ssec:spacefillingcurve}

In \cite{Poo04}, by applying his celebrated Bertini theorem over finite fields, Poonen proved the following result. 

\begin{theorem}[\protect{\cite[Corollary 3.5]{Poo04}}]\label{thm:Poonen}
Let $X$ be a smooth, projective, geometrically integral scheme of dimension $m \ge 1$ over $\FF_q$. Then there exists a smooth, projective, geometrically integral curve $C \subset X$ such that $C(\FF_q) = X(\FF_q)$. 
\end{theorem}

In other words, the curve $C$ passes through all $\FF_q$-points in $X$, no matter how large the dimension of $X$ is! 

Here we explain a few scheme theoretic terminologies in the statement of Theorem \ref{thm:Poonen}. For a scheme $X$ defined over $\FF_q$, $X_{\overline{\FF}_q}$ is the scheme defined by the same set of polynomials, but understood as polynomials in $\overline{\FF}_q[x_0, \cdots, x_n]$. A scheme $X$ over $\FF_q$ is \emph{geometrically integral} if $X_{\overline{\FF}_q}$ is integral, that is, irreducible and reduced. A scheme $X \subset \PP^n$ of dimension $d$ defined over $\FF_q$ is \emph{smooth} if for every point $x \in X$, the Jacobian matrix obtained by local equations is of rank $n-d$. For the details, see \cite[Definition 3.5.12]{Poo17}.

\begin{remark}\label{rmk:completeintersection}
Based on his approach, we can show that such $C$ can be constructed as a complete intersection in $X$ -- if $X \subset \PP^n$ and $\dim X = m$, then one can find $m - 1$ homogeneous polynomials $f_1, \cdots, f_{m-1} \in \FF_q[x_0, \cdots, x_n]$ such that $C = V(f_1, \cdots f_{m-1}) \cap X$. 
\end{remark}

Poonen's approach does not provide any explicit example. Here we focus on the construction of concrete examples of space filling curves. Before going further, here we fix our terminology. 

\begin{definition}\label{def:spacefillingcurve}
Fix a base field $\FF_q$. We say a curve $C \subset \PP^n$ is a \emph{space filling curve for $\PP^n$} if $C$ is a projective geometrically irreducible curve such that $C(\FF_q) = \PP^n(\FF_q)$. Additionally, if $C$ is smooth, we say $C$ is a \emph{smooth space filling curve}. 
\end{definition}

\begin{lemma}\label{lem:idealofpoints}
Let $\PP^n = \mathrm{Proj}\;\FF_q[x_0, \cdots, x_n]$, in other words, let $(x_0 : \cdots : x_n)$ be the homogeneous coordinates of $\PP^n$. The ideal of the set of all $\FF_q$-rational points $\PP^n(\FF_q)$ is 
\begin{equation}\label{eqn:idealJ}
    J = (x_i^q x_j - x_i x_j^q)_{0 \le i \ne j \le n}.
\end{equation}
\end{lemma}

\begin{proof}
Since every $a \in \FF_q$ satisfies $a^q = a$, it is routine to check that any generator of $J$ vanishes at $\PP^n(\FF_q)$. So we have $\PP^n(\FF_q) \subset V(J)$. 

Conversely, pick a point $\bp = (p_0 : \cdots : p_n) \in V(J)$. We may assume that $p_0 = 1$. Take a standard affine open neighborhood $x_0 \ne 0$. Because $\bp$ is a point in $V(J) \cap \AA^n(\FF_q)$, it is zero for all dehomogenizations of $x_i^q x_0 - x_i x_0^q$, which is $x_i^q - x_i$. In particular, its $i$-th coordinate is in $\FF_q$ for all $1 \le i \le n$. Hence $\bp \in \PP^n(\FF_q)$. Therefore, set theoretically, $V(J) = \PP^n(\FF_q)$.

It remains to show that they have the same scheme structures. Because $\PP^n(\FF_q)$ is a finite set of reduced points, it is sufficient to show that $V(J)$ is also reduced at each point. Essentially, it follows from the fact that $x_i^q - x_i = \prod_{a \in \FF_q}(x_i - a)$ and the right hand side has no multiple factor. Take $\bp = (p_0 : \cdots : p_n) \in V(J)$ with $p_0 = 1$ as before. On the affine chart $V(J) \cap \AA^n(\FF_q)$ given by $p_0 = 1$, take a localization along the union of hyperplanes of type $x_j - a$ which do not pass $\bp$. Then each defining equation $x_i^q - x_i$ is a unit times $x_i - p_i$. Therefore, after the localization, the ideal of $V(J) \cap \AA^n(\FF_q)$ is $(x_1 - p_1, \cdots, x_n - p_n)$, hence reduced at $\bp$.
\end{proof}

\begin{corollary}\label{cor:equations}
Suppose that $C \subset \PP^n$ is a space filling curve over $\FF_q$. Then any defining equation of $C$ must be in the ideal $J$ in \eqref{eqn:idealJ}. In particular, every defining equation of $C$ is of degree at least $q+1$.
\end{corollary}

\begin{proof}
If there is a space filling curve $C$, then from $C(\FF_q) = \PP^n(\FF_q)$, $\PP^n(\FF_q) \subset C$ as a scheme. This implies $I(C) \subset I(\PP^n(\FF_q)) = J$.    

The last assertion follows from the fact that every generator of $J$ has degree $q+1$.
\end{proof}

Moreover, the degree of $C$ has to be large as well. 

\begin{proposition}\label{prop:degreeofspacefillingcurve}
Let $C \subset \PP^n$ be a smooth space filling curve over $\FF_q$. Then $\deg C \ge (q^n-1)/(q-1)+1$. 
\end{proposition}
\begin{proof}
Pick an $\FF_q$-point $\bp \in C(\FF_q)$. Since $C$ is defined by polynomials with $\FF_q$-coefficients, its tangent line $\ell$ at $\bp$ is also defined over $\FF_q$, because its directional vector is in the kernel of an $\FF_q$-matrix generated by the gradient vectors of the defining equations of $C$. Hence we may find an $\FF_q$-hyperplane $H$ that contains $\ell$. By definition, $\deg C$ is the number of intersection points in $H \cap C$ counted with multiplicity. Since $C$ is a space filling curve, $(H \cap C)(\FF_q) = H(\FF_q)$. Moreover, $H$ and $C$ are tangent to each other at $\bp$, hence the multiplicity is at least two. Therefore, 
\[
    \deg C \ge |(H \cap C)(\FF_q)| + 1 = |H(\FF_q)| + 1 = \frac{q^n - 1}{q-1}+1.
\]
\end{proof}

For $\PP^3$, the inequality is specialized to $\deg C \ge q^2 + q + 2$.

\subsection{Plane filling curves}\label{ssec:plainfilingcurve}

Perhaps the simplest non-trivial case is that $X = \PP^2$. Homma and Kim showed that a classical example of Tallini \cite{Tal61} is indeed an example of a smooth plane filling curve of minimal degree by showing it is smooth \cite{HK13}.

Let $\PP^2$ be the projective plane over $\FF_q$. We set $\PP^2 = \mathrm{Proj} \; \FF_q[x, y, z]$. Homma and Kim showed that any curve of the form
\begin{equation}
	f_A := (x, y, z)A(y^qz-yz^q, z^qx-zx^q, x^qy-xy^q)^t
\end{equation}
for $A \in \mathrm{GL}_3(\FF_q)$, defines a smooth space filling curve $V(f_A)$, if and only if $A$ has the characteristic polynomial that is irreducible over $\FF_q$ \cite[Theorem 3.2]{HK13}. 

From the definition of $f_A$, it is clear that $f_A$ is in the ideal generated by $\{y^qz-yz^q, z^qx-zx^q, x^qy-xy^q\}$. In particular, by Lemma \ref{lem:idealofpoints}, $V(f_A)$ is a space filling curve. 

For our curve to be smooth, we need each gradient vector to be non-zero for every point in $V(f_A)$. Since we need to consider all points over any extension field of $\FF_q$, this is already a non-trivial task. Homma and Kim performed an explicit calculation of the partial derivatives and showed that there is no singular point on $V(f_A)$. Finally, because $V(f_A)$ is connected, the smoothness implies the irreducibility \cite[II.8.Ex.4.]{Har77}.

\section{Space filling surfaces}\label{sec:spacefillingsurface}

In this section, we describe the properties of a space filling surface and how its construction relate to finding a smooth space filling curve. 

Let $\FF_q$ be a finite field of order $q = p^r$. Let $\PP^3 = \mathrm{Proj}\; \FF_q[x, y, z, w]$. 

\subsection{Basic strategy}\label{ssec:strategy}

We seek our curve $C$ that is a complete intersection in $\PP^3$ (The existence follows from the proof of Theorem \ref{thm:Poonen}). Thus, $C = V(f_1, f_2) \subset \PP^3$. Moreover, since $V(f_1, f_2) = V(f_1) \cap V(f_2)$, each $V(f_i) \subset \PP^3$ must also satisfy $V(f_i)(\FF_q) = \PP^3(\FF_q)$. Therefore, we need to find two \emph{space filling surfaces} such that $V(f_1) \cap V(f_2)$ is a smooth space filling curve. 

\subsection{First example}\label{ssec:firstexample}

\begin{definition}\label{def:varphipoly} 
Let 
\[
    \varphi_{xy} := x^q y - x y^q.
\]
Similarly we can define $\varphi_{xz}, \varphi_{xw}, \cdots$. Note that $\varphi_{yx} = -\varphi_{xy}$.
\end{definition}

In this case, Lemma \ref{lem:idealofpoints} is interpreted as the following: The ideal $J$ in \eqref{eqn:idealJ} is 
\[
    J = (\varphi_{xy}, \varphi_{xz}, \varphi_{xw}, \varphi_{yz}, \varphi_{yw}, \varphi_{zw}).
\]

\begin{lemma}\label{lem:spacefillingsurfacep+1}
Let $g := \varphi_{xy} + \varphi_{zw}$. Then $V(g)$ is a space filling surface of minimal degree $q+1$. 
\end{lemma}

\begin{proof}
Note that $\nabla g = (-y^q, x^q, -w^q, z^q)$. Thus, this vector is nonzero at every point (not only $\FF_q$-rational points!). Therefore, $V(g)$ is a smooth variety. Because $V(g)$ is connected by \cite[II.8.Ex.4.]{Har77}, the smoothness implies the irreducibility. Since the minimal degree of elements in $J$ is $q+1$, $V(g)$ is a space filling surface with minimal degree.
\end{proof}

\begin{remark}\label{rmk:odddimensionalspace}
The same proof shows that for any odd dimensional projective space $\PP^n$, there is a smooth space filling hypersurface of degree $q+1$. Indeed, one can take 
\[
    g = \sum_{0 \le i \le n, 2|i} \varphi_{x_i x_{i+1}},
\]
and the same proof shows that $V(g) \subset \PP^n$ is a smooth space filling hypersurface. 

The same construction does not work for an even dimensional projective space. Indeed, Tallini showed that any irreducible degree $q+1$ space filling curve in $\PP^2$ has a unique singular point.
\end{remark}

\subsection{Higher degree examples}\label{ssec:higherdegree}

Now we need to find another space filling surface to find a space filling curve. One may ask if we can find another surface of degree $q+1$ and take the intersection of them. Indeed, this is not possible. 

In \cite{Hom12}, improving the Serre-Weil bound 
\[
    |C(\FF_q)| \le q + 1 + 2g\sqrt{q}, 
\]
(here $g$ is the genus of $C$), Homma showed the following bound of the number of $\FF_q$-rational points on $C$. 
\begin{theorem}[\protect{\cite[Theorem 3.2]{Hom12}}]\label{thm:Homma}
Let $C$ be a non-degenerate irreducible curve of degree $d$ in $\PP^3$ over $\FF_q$. Then 
\begin{equation}
    |C(\FF_q)| \le \frac{(q-1)|\PP^3(\FF_q)|}{q^3+q^2+q-3}d.
\end{equation}
\end{theorem}

\begin{corollary}\label{cor:minimumdegree}
The degree of a smooth complete intersection space filling curve in $\PP^3$ is at least $(q+1)(q+2)$.
\end{corollary}

\begin{proof}
Suppose that $C$ is a smooth complete intersection space filling curve in $\PP^3$. Any space filling curve is non-degenerate. By Theorem \ref{thm:Homma}, 
\begin{equation}\label{eqn:Hommabound}
    |\PP^3(\FF_q)| = |C(\FF_q)| \le \frac{(q-1)|\PP^3(\FF_q)|}{q^3+q^2+q-3}d.
\end{equation}
Since $C = V(f_1, f_2)$, $d = \deg f_1 \cdot \deg f_2$ by B\'ezout's theorem \cite[Proposition 8.4]{Ful98}. Moreover, $\deg f_i \ge q+1$ by Corollary \ref{cor:equations}. One can see that $\deg f_1 = \deg f_2 = q+1$ violates \eqref{eqn:Hommabound}.
\end{proof}

\begin{question}\label{que:minimumdegree}
Can we find a smooth space filling curve $C$ of degree $(q+1)(q+2)$?
\end{question}

Let $f \in \FF_q[x, y, z, w]$ be the homogeneous polynomial defining the second surface $V(f)$. We require that: 
\begin{enumerate}
    \item $f \in J$ by Corollary \ref{cor:equations};
    \item $\deg f \ge q+2$;
    \item $V(g, f) \subset \PP^3$ is smooth. 
\end{enumerate}
If we find such $f \in \FF_q[x, y, z, w]$, then $C := V(g, f)$ is a space filling curve of degree $(q+1)\deg f$. Here $g$ is the degree $q+1$ polynomial in Lemma \ref{lem:spacefillingsurfacep+1}.

The last assertion can be checked as the following. At a point $\bp \in C = V(g, f)$, $\bp$ is a singular point if and only if the gradient vectors $\nabla g(\bp)$ and $\nabla f(\bp)$ are linearly dependent. Thus, if we set the $2 \times 4$ matrix
\begin{equation}\label{eqn:Mmatrix}
    M(\bp) := \begin{bmatrix}\nabla g(\bp) \\ \nabla f(\bp)\end{bmatrix}
\end{equation}
and let $D_{ij}(\bp)$ be the determinant of the $2 \times 2$ minor of $M(\bp)$ obtained by taking $i$-th and $j$-th columns, then $C$ is smooth if and only if $V(g, f, D_{ij}) = \emptyset$, or equivalently, $(g, f, D_{ij}) = \FF_q[x,y,z,w]$.

\section{Numerical examples}\label{sec:numericalexamples}

By using random constructions with \texttt{Macaulay2}, we are able to identify multiple low degree space filling curves for $\FF_q$ with small $q$, and thus, to prove Theorem \ref{thm:mainthm}. The code used to randomly generate such curves is as follows:
\begin{verbatim}
needsPackage("SpaceCurves");
p = 3;
r = 1;
q = p^r;
K = GF(q, Variable => a);
d = 1;
S = K[x, y, z, w];
g = x^q*y-x*y^q+z^q*w-z*w^q;
J = ideal(x^q*y-x*y^q, x^q*z-x*z^q, x^q*w-x*w^q, 
y^q*z-y*z^q, y^q*w-y*w^q, z^q*w-z*w^q);
found = false;
counter = 0;
while not found do
    {
    f = random(q+1+d, J);
    I = ideal(f, g);
    if isSmooth(I) and dim I == 2 then 
    {
        print f;
        print counter;
        found = true;
    }
    else counter = counter + 1;
    }
\end{verbatim}

In the above code, for a randomly chosen polynomial $f$ in the ideal $J$, we compute the ideal $I$ generated by $f$ and $g$. We check if $C = V(f, g) = V(I)$ is smooth, and it is of dimension one. Since $\dim I$ is the affine dimension, $\dim I = \dim C + 1$. 

Table \ref{tbl:spacefillingcurve} shows examples of polynomials $f \in \FF_q[x, y, z, w]$ such that $C = V(g, f) \subset \PP^3$ is a space filling curve over $\FF_q$.

\begin{table}[!ht]
\begin{tabular}{|l|l|}\hline
$q$ & $f$\\ \hline \hline
$2$ & $(x+z)\varphi_{xz} + (y+z+w)\varphi_{xw} + (x+y+z+w)\varphi_{yz}+ (z+w)\varphi_{yw} + y\varphi_{zw}$\\ \hline
$3$ & $(-x)\varphi_{xy} + (x+y)\varphi_{xz} + (x+y-z)\varphi_{xw}$\\
& $ + (x-y-z)\varphi_{yz} + (x-w)\varphi_{yw} + (-x-y-z-w)\varphi_{zw}$\\ \hline
$4$ & $(ax + y + (a + 1)z + w)\varphi_{xy} + (ax + z + aw)\varphi_{xz}+(ax + y + (a + 1)z + aw)\varphi_{xw}$\\
& $+(ax + y + z + aw)\varphi_{yz} + (y + (a + 1)z + w)\varphi_{yw} + (x + y + az)\varphi_{zw}$\\ \hline
$5$ & $(-x+2y+2z)\varphi_{xy}+(x+3z-2w)\varphi_{xz}+(-x+w)\varphi_{xw}$\\
& $+ (-2x-y-2z-w)\varphi_{yz}+(-2x+y+w)\varphi_{yw} +(2x+2y+z+w)\varphi_{zw}$\\ \hline
$7$ & $(- x - y - z - 3w)\varphi_{xy} + (z - 2w)\varphi_{xz} + (- x - 2y + z + w)\varphi_{xw}$\\
& $+ (3x + y)\varphi_{yz} + (- x - 2y + z - w)\varphi_{yw} + z\varphi_{zw}$\\ \hline
\end{tabular}
\caption{Examples of polynomials defining space filling curves. For a field of non-prime order, $a$ is a cyclic generator of $\FF_q^*$.}\label{tbl:spacefillingcurve}
\end{table}

\begin{remark}\label{rmk:rarity}
As one may guess, the polynomials that induce space filling curves are very rare. For $q = 2$, approximately only $0.264\%$ of random sample polynomials define space filling curves. We obtained examples in Table \ref{tbl:spacefillingcurve} after many failed attempts. The example for $q = 3$ was obtained after $1,482$ tries, and the $q = 7$ examples was obtained after $29,923$ tries. For $q = 8$, $100,000$ tries did not make an example.
\end{remark}

\begin{remark}\label{rmk:largerdegreespacefillingcurves}   
The similar random construction method can be used to construct a smooth complete intersection space filling curve with larger degree. Numerically, we observed that, at least for small $q$, we obtained examples of a smooth space filling curve faster.
\end{remark}


\section{Toward the existence of space filling curves of low degree}\label{sec:existence}

Unfortunately, we were unable to answer Question \ref{que:minimumdegree} completely. In this section, we leave our approach and challenge, and why we expect such a low degree space filling curve exists. 

In Lemma \ref{lem:spacefillingsurfacep+1}, we showed that $g := \varphi_{xy} + \varphi_{zw}$ defines a space filling surface $V(g)$ of degree $q+1$. We need to show that there is a space filling surface $V(f)$ of degree $q+2$ such that $C := V(f, g)$ is smooth curve. Since $V(f, g)$ is a complete intersection, it is connected \cite[II.8.Ex.4]{Har77}, hence the irreducibility follows. 

Such an $f$ must be a degree $p+2$ polynomial in $J$. Thus, there are linear polynomials $\ell^{xy}, \cdots, \ell^{zw} \in \FF_q[x, y, z, w]_1$ such that 
\[
    f = \ell^{xy}\phi_{xy} + \ell^{xz} \phi_{xz} + \cdots + \ell^{zw} \phi_{zw}.
\]
Let $\ell^{xy} = a^{xy}x + b^{xy}y + c^{xy}z + d^{xy}w$, and so on. Then we can understand $f$ as an element of
\[
    \FF_q[a^{xy}, b^{xy}, \cdots, d^{zw}, x, y, z, w].
\]
Furthermore, if we divide the set of variables into two sets $\{a^{xy}, \cdots, d^{zw}\} \sqcup \{x, y, z, w\}$, $g$ is a bihomogeneous of degree $(1, q+2)$. Thus, $f$ defines a hypersurface of bidegree $(1, q+2)$ in $\PP^{23} \times \PP^3$, where $\PP^{23}$ is a projective space with homogeneous coordinates $(a^{xy}: b^{xy}: \cdots : d^{zw})$ and $\PP^3$ is a projective space with homogeneous coordinates $(x:  y : z: w)$. We may understand $g$ as a bihomogeneous polynomial of bidegree $(0, q+1)$. 

Take the gradient vectors $\nabla f$ and $\nabla g$ of $f$ and $g$, with respect to $(x, y, z, w)$. Let $M$ be a $2 \times 4$ matrix in \eqref{eqn:Mmatrix} and $D_{ij}$ be the determinant of $2 \times 2$ minors of $i$-th and $j$-th columns. Then $D_{ij}$ is also a bihomogeneous polynomial of degree $(1, 2q+1)$. 

For any point (not necessarily an $\FF_q$-rational point) $\bv = (a^{xy} : \cdots : d^{zw}) \in \PP^{23}$, let $f_\bv \in \FF_q[x, y, z, w]$ be a polynomial obtained by evaluating $\bv$, that is, $f_\bv(x, y, z, w) := f(\bv, x, y, z, w)$. Let $U = V(f, g, D_{ij}) \subset \PP^{23} \times \PP^3$ be a biprojective scheme defined by $f$, $g$, and $D_{ij}$. For the projection $\pi : \PP^{23} \times \PP^3 \to \PP^{23}$ to the first factor, let $W := \pi(U) \subset \PP^{23}$. Then $U$ and $W$ have the following geometric interpretation. First of all, we may interpret $\PP^{23}$ as a parameter space of surfaces of degree $q+2$ that interpolates all points in $\PP^3(\FF_q)$, as every point $\bv = (a^{xy}, \cdots, d^{zw}) \in \PP^{23}$ defines such a surface $V(f_\bv) \subset \PP^3$ over some extension field $K$ of $\FF_q$. (We do not use the word space filling surface here, because it requires that every coefficient of the defining equations must be in $\FF_q$.) Thus, for $\bv := (a^{xy}, \cdots, d^{zw}) \in W$, its fiber $\pi^{-1}(\bv)$ is naturally identified with $V(f_\bv, g, D_{ij}) \subset \PP^3$, which is precisely the singular locus of $V(f_\bv, g)$. In other words, $\bv \in W = \pi(U)$ if and only if the scheme $V(f_\bv, g)$ is singular. So there is a nonsingular space filling curve $V(f_\bv, g)$, defined over $\FF_q$, if and only if $W(\FF_q) \ne \PP^{23}(\FF_q)$. In other words, if $W$ is not a space filling scheme!

From the Jacobian computation, one can check that $\mathrm{codim}\; U = 7$ (The number of equations is 8, but $\{D_{ij}\}$ are not algebraically independent -- there is one Pl\"ucker relation between them). Since there is a complete intersection $V(f_\bv, g)$ with one isolated singularity, the map $\pi : U \to W$ is a birational map. Therefore, $W$ is a proper subscheme of $\PP^{23}$ of codimension $4$.

As we may guess, space filling schemes are very rare. The probability that we have a smooth space filling complete intersections were evaluated in \cite[Theorem 1.2]{Poo04} and \cite[Corollary 1.3]{BK12} and its formula involves the zeta function of the ambient projective space. The only non asymptotic example we are aware of is the case of $0$-dimensional scheme in $\AA^1$. In \cite{JMW23}, the probability distribution of number of $\FF_q$-points of a random polynomial $f$ is achieved. As a consequence, we can show that the probability that a randomly chosen scheme $V(f)$ defined by a degree $d \ge q$ polynomial is a space filling scheme is $1/q^q$. See also Remark \ref{rmk:rarity}. So we may expect, though we do not have a rigorous proof, that $W$ is not a space filling scheme and there is an example of a space filling curve of minimal degree.

To show that $W$ is not a space filling scheme, we have tried several approaches. First of all, note that any space filling hypersurface must be of degree at least $q+1$ by Lemma~\ref{lem:idealofpoints}. Thus, an affirmative answer to the following question implies the existence of a space filling curve of degree $(q+1)(q+2)$.

\begin{question}
Can we find any degree $d \le q$ homogeneous polynomial $h \in \FF_q[a^{xy}, \cdots, d^{zw}]$ such that $W \subset V(h)$?
\end{question}

If $W$ is a space filling scheme and $L$ is an $\FF_q$-linear subspace of $4$-dimensional subspace of $\PP^{23}$ that intersects $W$ properly (But we do not know such a subspace exists!), $\deg W \ge |(L \cap W)(\FF_q)| = |L(\FF_q)| = (q^5-1)/(q-1)$. Since $W = \pi(U)$ and the map $\pi : U \to W$ is birational, we may compute the degree of $W$ from the multidegrees of $U$. This is not entirely obvious because $U$ is not a complete intersection in $\PP^{23} \times \PP^3$. But we expect that the degree is larger than $(q^5-1)/(q-1)$, from the degree formula of complete intersections in a biprojective space \cite[Example 8.4.2]{Ful98}, so this approach might be inconclusive. 

\begin{question}
Another possible approach is to find a formula of $f$ that works for arbitrary field $\FF_q$. Can we find such a nice formula?
\end{question}

\end{document}